\documentclass{lmcs}

\usepackage{hyperref}
\input{base}
\usepackage[inline]{enumitem}

\theoremstyle{definition}

\usetikzlibrary{decorations.pathreplacing}
\newcommand{\IVTast}{IVT\(^\ast\)}

\begin{document}

\title[(Seemingly) Impossible Theorems]{(Seemingly) Impossible Theorems in Constructive Mathematics}

\author[H.~Diener]{Hannes Diener}	%required
\address{School of Mathematics and Statistics, University of Canterbury, Christchurch, New Zealand}	%required
\email{hannes.diener@canterbury.ac.nz}  %optional
%\thanks{thanks 1, optional.}	%optional

\author[M.~Hendtlass]{Matthew Hendtlass}	%optional
\address{School of Mathematics and Statistics, University of Canterbury, Christchurch, New Zealand}
\email{matthew.hendtlass@canterbury.ac.nz}  %optional

%% etc.

%% required for running head on odd and even pages, use suitable
%% abbreviations in case of long titles and many authors:

%%%%%%%%%%%%%%%%%%%%%%%%%%%%%%%%%%%%%%%%%%%%%%%%%%%%%%%%%%%%%%%%%%%%%%%%%%%

%% the abstract has to PRECEDE the command \maketitle:
%% be sure not to issue the \maketitle command twice!

\begin{abstract}
  \noindent 	We prove some constructive results that on first and maybe even on second glance seem impossible.
\end{abstract}

\maketitle

\section{Introduction}

One of the adjustments that all practitioners of constructive mathematics  coming from classical mathematics  (CLASS) have to make is not to rely on true-false case distinctions, such as 
\[ \fa{x \in \RR}{x = 0 \lor x  \neq 0}  \ ; \]
which in the classical treatment are simply an appeal to the law of excluded middle (\LEM). Not only are these fragments of \LEM unavailable to the constructive mathematician, they are often outright false in many varieties of constructive mathematics, such as Brouwer's intuitionism (INT), Russian Recursive Mathematics (RUSS), or realizability models. One might easily gain the impression that it is altogether impossible to prove any interesting, general, and non-trivial disjunction over intuitionistic logic. This is what makes the collection of results in this paper intriguing, since they are all of the restricted \LEM-form 
\[ \fa{x \in \mathbb{D}}{A(x) \lor \lnot A(x)}  \ . \]
We  refer to these theorems as ``impossible theorems''. Of course they are not actually impossible, merely surprising. There does not seem to be a master lemma from which all of our impossible theorems follow, but they all share the same basic proof structure, a kind of ``positive diagonalisation'': a function or assertion is applied to yield information that is used to build another object, this object is then fed into the original function or assertion to ``force'' it to give up the information we desire. This resembles diagonalisation, with the difference that the goal is not to construct an object that produces an absurdity, but rather one that produces positive information.\footnote{One could argue that the solutions of Smullyan's well known ``Knights and Knaves'' puzzles \cite{rS90}  also require a sort of positive diagonalisation. } This technique has a rich history in constructive mathematics and can be found in the proofs of many named theorems  (see Section 12 of \cite{mE13b} for a quick overview). The algorithms embodied in these proofs have been dubbed ``(seemingly) impossible programs'' \cite{aB07}, and our present work should be seen as a continuation of that tradition. 

The results we prove are the following.
\begin{itemize}
  \item An improvement on Ishihara's second trick, which interpreted very loosely says that, under some reasonable conditions, we can decide whether a function is sequentially continuous or not.
  \item A version of the approximate intermediate value theorem that does not assume continuity; together with some consequences.
  \item A result on decomposing sets.
  \item Implications of a set being neatly located.
  \item A theorem that allows us to decide whether the infimum of a quasi-convex, inf-located function, point-wise positive function is positive or zero.
\end{itemize}

These results are delightful in their own right and generalise the main results of \cite{hI91}, \cite{dB01}, and \cite{jB16}, but apart from this, we see our work as furthering constructive mathematics in a few, more general, ways.
\begin{itemize}
\item One reason that Bishop was able to develop large parts of analysis constructively \cite{eB67} without getting caught up in philosophical debates---as Brouwer and his followers were before him  (``Grundlagenstreit'')---was that he largely leapfrogged foundational issues by considering only very restricted classes of objects. For example in \cite{eB67} Bishop only considered functions \(\RR \to \RR\) that are uniformly continuous on compact sets\footnote{And---in a cunning marketing move---he simply called continuous.} and never dealt with other ones. He thereby avoided any philosophical argument of whether, say, bar-induction is valid, or whether all functions are continuous. While this was unquestionably a tremendously successful way of getting things done, it left large classes of objects under-investigated, such as  point-wise continuous functions, or functions that do not inherently come with a modulus of continuity. 
All of our results are only interesting when applied to those larger domains.
 By doing this, we not only extend existing results, we more importantly  also clarify which properties actually matter.
\item This reduction of assumptions also means that results are not just applicable in traditional varieties of mathematics (CLASS, INT, RUSS), but  also apply to more exotic and harder to formalise scenarios. For example we may imagine a function being computed by a black-box (say in the cloud). Even if we do not know how the function is computed (maybe with random inputs, maybe with human intervention), we are still able to extract some meta-information about the function without ``opening up the box''. 
\item Returning to more traditional varieties, our results also allow us to prove many results \emph{implicitly} from weak assumptions rather than \emph{explicitly} from strong principles. For example, the fact that all functions \([0,1] \to \RR\) in INT are sequentially continuous can be deduced from \(\neg \LPO\) and \(\WMP\) via Ishihara's tricks \cite{hI91,hD17}, rather than from principles that come with more philosophical baggage, such as continuous choice. The same proof will also work in  RUSS, and without assuming that we know for each function \([0,1] \to \RR\) the G\"odel number of a Turing Machine which realizes it.

Another example is Corollary \ref{Cor:epssteps} (an increasing function \([0,1] \to \RR\) can only make finitely many \(\varepsilon\)-steps for a given \(\varepsilon >0\)), which is easily provable in both classical mathematics and in ``truly'' constructive schools (CON) of mathematics (i.e.\ those in which \LPO fails). However direct proofs in both cases would be greatly different: In CLASS we can use a simple argument that the set of \(\varepsilon\)-steps has to be finite (otherwise the increasing function would grow too much), whereas in CON the result is true, since, there, \emph{all} functions are continuous, and  all increasing, point-wise continuous functions on the unit interval are automatically uniformly continuous \cite[Theorem 13.6]{mM83}. Our proof works in all situations, and without assuming \LPO or its negation.
\end{itemize}

\section{Technical Preliminaries}

We are working in Bishop style constructive mathematics (BISH)---that is, mathematics using intuitionistic logic and an appropriate set-theoretic or type-theoretic foundation such as \cite{pA01}. In the tradition of Bishop we are favouring clarity of presentation over rigid formalisms, but we do at the same time believe that it is merely a mechanical exercise to formalise our results in an appropriate setting. Not choosing a particular formal framework also means that our results are transferrable to a wide range of varieties of mathematics---even non-traditional ones---as outlined above.

We will use countable and dependent choice, but we will do so explicitly. 

In \cite{hD17} the authors analysed a construction used in many clever constructive results. It also features repeatedly in this paper, starting with the first lemma below, so it seems worth re-introducing the notation.
\begin{Definition}
	Let \( X\) be a set.  For a sequence \( x = (x_{n})_{n \geqslant 1} \) in \( X \), a point \( x_{\infty} \in X \), and an increasing binary sequence \( \lambda = (\lambda_{n})_{n \geqslant 1} \), we define a sequence \( \lambda \circledast x  \)   by
 \[ 
   (\lambda \circledast x )_{n} = \begin{cases}  
               x_{m} & \text{if } \lambda_{n} = 1 \text{ and } \lambda_{m}=1- \lambda_{m+1}  \\ 
               x_{\infty} & \text{if } \lambda_{n} = 0  \ .
                                                    \end{cases}
 \] 
\end{Definition}
If \( (X, \rho) \) is a metric space and \( x_n \to x_\infty\), then it is easy to show that for any increasing binary sequence \( \lambda \) the sequence \( \lambda \circledast x  \) is automatically a  Cauchy sequence \cite[Lemma 2.1]{hD12b}.

We make frequent reference to two classically valid ``omniscience principles'', the limited principle of omniscience (\LPO) and the weak limited principle of omniscience (\WLPO).
 \begin{itemize}
  \item[] \LPO: \label{PR:LPO} for every binary sequence \( (a_n)_{n \geqslant 1} \) we have \[ \fa{n \in \NN}{a_n =0} \lor \ex{n \in \NN}{a_n = 1} \ . \]
  \item[] \WLPO: \label{PR:WLPO} for every binary sequence \( (a_n)_{n \geqslant 1} \) we have \[ \fa{n \in \NN}{a_n =0} \lor \neg\fa{n \in \NN}{a_n = 0} \ .\]
 \end{itemize}
In the presence of countable choice \LPO is equivalent to \(\fa{x \in \RR}{x < 0 \lor x = 0 \lor x > 0}\) and \WLPO is equivalent to \[\fa{x \in \RR}{x = 0 \lor \neg(x = 0) } \ . \]

The following lemma is very much folklore, at least the \WLPO part, however we were unable to find it in the literature. It is very similar to its counterpart for functions \(\BS \to \NN\) as proven in \cite{mA02}. Here and in the following a function \(f:X \to Y\) between two metric \((X, \sigma)\) and \((Y, \rho)\) is called \define{strongly extensional} if \(\fa{x,y \in X}{f(x) \neq f(y) \implies x \neq y}\), or to be more precise, \[ \fa{x,y \in X}{\rho(f(x), f(y)) > 0 \implies \sigma(x, y) > 0}\ . \]

\begin{Lemma} \label{Lem:DiscImplOmniscience} 
 The existence of a discontinuous\footnote{We call a function \(f:X \to Y \) between metric spaces \define{discontinuous}, if there exists \(x_n \to x\) and \(\varepsilon >0\) such that \(\rho(f(x_n),f(x)) > \varepsilon \) for all \( n \).} function \(f:[0,1] \to \RR \) is equivalent to  \WLPO. The existence of a strongly extensional discontinuous function \(f:[0,1] \to \RR \) is equivalent to  \LPO. 
\end{Lemma}
\begin{proof}
Under the assumption of countable choice, if \WLPO holds,  we have \(\fa{x \in \RR}{x = 0 \ \lor \ \neg(x=0)}\), so the discontinuous function \(f:[0,1] \to \RR\) given by 
	\[ f(x) = \begin{cases}
		0 & \text{if } x = 0 \\ 1 & \text{if } \neg (x=0) 
	\end{cases} \]
	is well-defined. If in addition \LPO holds, it is also easily seen to be strongly extensional. 
	
Conversely assume that \( f: [0,1] \to \RR \) is discontinuous: there exist \(x_n \to x\) and \(\varepsilon >0\) such that \(\abs*{f(x_n) - f(x)} > \varepsilon \) for all \( n \). Furthermore let \((\lambda_n)_{n \geqslant 1}\) be an increasing binary sequence. Using the notation introduced above, consider the Cauchy sequence \( (\lambda_n) \circledast (x_n) \), and let \(z \in [0,1] \) be its limit. Then \(\abs*{f(z) - f(x)} < \varepsilon \) or \(\abs*{f(z) - f(x)} > 0 \).

In the first case we must have \(\fa{n \in \NN}{\lambda_n = 0}\), for if \(\lambda_n = 1\) for some \( n \), then \(z = x_m \), where \(\lambda_m = 1-\lambda_{m+1}\), and so \(\abs*{f(z) - f(x)} =  \abs*{f(x_m) - f(x)} < \varepsilon \), which is absurd. In the second case the assumption that \(\fa{n \in \NN}{\lambda_n = 0}\) implies that \(z = x \); it follows from this contradiction that \(\neg \fa{n \in \NN}{\lambda_n = 0}\). Hence \WLPO holds. If in addition \(f\) is strongly extensional, then in that second case \(\abs*{x - z} > 0 \). Now choose \(N\) such that \( \abs*{x - x_n} < \abs*{x - z} \) for all \(n \geqslant N\). For that \(N\) we must have \(\lambda_N=1\). Thus if \(f\) is strongly extensional, \LPO holds.
\end{proof}

Finally, note that by \define{increasing} we mean non-decreasing, rather than strictly increasing.

\section{Ishihara's Tricks}

The results that are nowadays known as Ishihara's tricks \cite{hI91} are both assuming a strongly extensional mapping \(f\) of a complete metric space \(X\) into a metric space \(Y\), and a sequence \((x_{n})_{n \geqslant 1}\)   in \(X\) converging to a limit \(x\).
\begin{Proposition}[Ishihara's first trick] For all positive reals \(\alpha < \beta\), 
 \begin{equation} \label{Eqn:IshFirst}
 	\ex{n \in \NN}{\rho(f(x_{n}),f(x)) > \alpha} \  \lor \  \fa{n \in \NN}{\rho(f(x_{n}),f(x)) < \beta} \ .
 \end{equation} 
\end{Proposition}
\begin{Proposition}[Ishihara's second trick] For all positive reals \(\alpha < \beta\), either we have 
\begin{itemize}
  \item \(\rho(f(x_{n}),f(x)) < \beta\) eventually,
  \item  or \(\rho(f(x_{n}),f(x)) > \alpha\) infinitely often.
\end{itemize}
\end{Proposition}

There are various results that improve, generalise, or modify Ishihara's tricks, such as \cite{dB03, hD12b, dBlV06}, and we would like to add two  improvements here. The first improvement, which is almost purely aesthetic, is that we can remove the assumption that \(0 < \alpha\) from the first and therefore also from the second trick.
\begin{proof}
	Simply choose \(\gamma\) such that \(\alpha < \gamma < \beta\). Now either \( 0 < \gamma \), or \( \alpha < 0 \). In the second case we have \( \rho(f(x_1),f(x)) \geqslant 0  > \alpha \), and we are done. In the first case we can continue along the lines of the usual proof to get that 
	 \[ \ex{n \in \NN}{\rho(f(x_{n}),f(x)) > \gamma} \  \lor \  \fa{n \in \NN}{\rho(f(x_{n}),f(x)) < \beta} \ . \] 
	 Since \( \gamma > \alpha\) we also have \eqref{Eqn:IshFirst}.
\end{proof}
Our second, more substantial, improvement is that we can make the  disjunction in Ishihara's second trick exclusive. We will use the same idea two more times throughout this paper (Corollary \ref{Cor:IVTQ} and Proposition \ref{Pro:impconvex}).
First, as has long been known, if the second alternative holds in Ishihara's second trick, then \LPO holds.\footnote{The first time, that we are aware of, that this has been explicitly stated is in \cite{dB01}.} 
Since \LPO allows one to check for a given \(n\) whether \(\rho(f(x_{n}),f(x)) < \beta\) or \(\rho(f(x_{n}),f(x)) \geqslant \beta\)  and since \LPO  also allows one to check whether the first alternative holds eventually or the second alternative holds infinitely often, we get the following.

\begin{Proposition}[Ishihara's second trick; slightly improved] \label{Pro:IshTrickImproved} For all  \( \beta > 0\), either \(\rho(f(x_{n}),f(x)) < \beta\) eventually or \(\rho(f(x_{n}),f(x)) \geqslant \beta \) infinitely often.
\end{Proposition}

\section{The Intermediate Value Theorem without Continuity}

It is well known that constructively there is no hope to prove the classical intermediate value theorem, but that finding approximate roots is fine.\footnote{By approximate roots we mean that for every \(\varepsilon > 0\) there exists \(x\) such that \(\abs*{f(x)} < \varepsilon\).} In what may seem like a strange move on first glance, but we can actually remove the continuity assumption and still have a constructively useful result.

\begin{Proposition}[\IVTast] \label{Pro:IVTast}
Let \(f:[0,1] \to \RR \) be such that \(f(0)\cdot f(1) \leqslant 0 \)	and let \(\varepsilon > 0\). There exist \(z \in [0,1] \) such that either \( \abs*{f(z)} < \varepsilon \) or there exists \( x_n \to z\) such that \(f(z)\cdot f(x_n) < -\nicefrac{\varepsilon}{8} \) for all \(n\). \\
Furthermore, in case the second alternative holds, \WLPO holds; and if \(f\) is strongly extensional, then \LPO holds.
\end{Proposition}
\begin{proof}
First, we may assume that both \(\abs*{f(0)} > \nicefrac{\varepsilon}{2}\) and \(\abs*{f(1)} > \nicefrac{\varepsilon}{2}\), since if either  \(\abs*{f(0)} < \varepsilon\) or \(\abs*{f(1)} < \varepsilon\) we are done. Without loss of generality, \(f(0) \leqslant 0\) and \(f(1) \geqslant 0\).
	As with the usual interval halving technique, construct, using dependent choice,\footnote{Actually countable choice is enough here: using  countable choice one can, at the outset, obtain a binary sequence \((\sigma_n)_{n \geqslant 1}\) such that \(\sigma_n = 0\) implies \(f(r_n) < \varepsilon\) and \(\sigma_n = 1\) implies \(f(r_n) > \nicefrac{\varepsilon}{2}\), where \(r_n\) is an enumeration of all rational numbers in \([0, 1]\). Since for every rational \(q\) we can find the minimal \(n\) such that \(r_n = q\) and during the interval halving procedure we only ever check the function value at rational points, the sequence \((\sigma_n)\) provides us with enough information to carry out the construction.} an increasing binary sequence \((\lambda_n)_{n \geqslant 1}\) and sequences \( (a_n)_{n \geqslant 1} \) and \((b_n)_{n \geqslant 1}\) such that for all \(n \in \NN\)
\begin{itemize}
  \item \(a_n \leqslant a_{n+1} \leqslant b_{n+1} \leqslant b_n \ , \)
  \item \( \abs*{b_n -a_n} \leqslant 2^{-n} \ , \)
  \item \( \lambda_n = 1 \implies a_n = b_n \ , \)
  \item \( \lambda_n = 0 \implies f(a_n) < - \nicefrac{\varepsilon}{2} \ \land \ f(b_n) > \nicefrac{\varepsilon}{2} \ . \)
\end{itemize}
It follows from the first two points that \((a_n)_{n \geqslant 1}\) and \((b_n)_{n \geqslant 1}\) are Cauchy sequences, which converge to the same limit \(z \in [0,1] \). Either \(\abs*{f(z)} < \nicefrac{\varepsilon}{2} \) and we are done, or \(f(z) > \nicefrac{\varepsilon}{4} \), or \(f(z) < -\nicefrac{\varepsilon}{4} \). In the second and third case, we must have \(\lambda_n = 0\) for all \( n \in \NN\), so either \((a_n)_{n \geqslant 1}\) or \((b_n)_{n \geqslant 1}\) respectively will be a witness of discontinuity, which implies \WLPO---and even \LPO if \(f\) is strongly extensional---as shown in Lemma \ref{Lem:DiscImplOmniscience}.
\end{proof}

Using the same technique as described above to prove Proposition \ref{Pro:IshTrickImproved}, the following corollary is somehow even more striking than the stronger proposition it relies on.
\begin{Corollary} \label{Cor:IVTQ}
Let \(f:[0,1] \to \RR \) be strongly extensional with \(f(0)\cdot f(1) \leqslant 0 \) and let \(\varepsilon > 0\). Either there exist \(z \in  [0,1] \) such that \( \abs*{f(z)} < \varepsilon \) or \(f(q) \geqslant \varepsilon \) for all \( q \in [0,1] \cap \QQ\).
\end{Corollary}
\begin{proof}
By the previous proposition, either there exist \(z \in [0,1] \) such that \( \abs*{f(z)} < \varepsilon \) or \LPO holds. In that second case we can fix a binary sequence \((\lambda_n)_{n \geqslant 1}\) such that  
\begin{align*}
	\lambda_n = 1 & \implies f(q_n) < \varepsilon \\
	\lambda_n = 0 & \implies f(q_n) \geqslant \varepsilon
\end{align*}
where \((q_n)_{n \geqslant 1}\) is an enumeration of \(\QQ \cap [0,1]\) (notice that the construction of \((\lambda_n)_{n \geqslant 1}\)  only needs unique choice). Now applying \LPO either \( \fa{n \in \NN}{\lambda_n = 0} \) and we are done, or \( \ex{n \in \NN}{\lambda_n = 1} \), in which case we are also done. 
\end{proof}

\section{Enumerating  Points of Discontinuity}
Aa a corollary of \IVTast (Proposition \ref{Pro:IVTast}) we get the following helpful little result.
\begin{Corollary}
	If a monotone, strongly extensional function \(f:[0,1] \to \RR\) is not discontinuous at \(x \in \RR \), then it is continuous at \(x\).
\end{Corollary}
\begin{proof}
	Consider an arbitrary \(\varepsilon >0\). We will first show that there is \(\delta > 0\) such that if \(y \in [0,1]\) is such that \(x \leqslant y \leqslant x+\delta\) then \(f(y) - f(x) < \varepsilon\). 
	
	We may assume that \(\abs*{f(x) - f(1)} = f(1) - f(x) > \nicefrac{\varepsilon}{2}\), since if \(f(1) - f(x) < \varepsilon\), then \emph{any} \(\delta > 0\) will do. By \IVTast there exists \(z \in [x,1]\) such that \[ f(z) \in \left[f(x) + \nicefrac{\varepsilon}{3}, f(x) + \nicefrac{\varepsilon}{2}\right] \ . \] Since \(f\) is strongly extensional \(z \neq x\), and in particular \(\delta = z-x>0\). Now let \(y \in [0,1]\)  such that \(x \leqslant y \leqslant x+\delta\). Then \[ f(y) - f(x) \leqslant f(z) - f(x) \leqslant f(x)+ \nicefrac{\varepsilon}{2} - f(x)  = \nicefrac{\varepsilon}{2} \ . \]
	Similarly we can find \(\delta^\prime\) which works for \(y \leqslant x\). Set \(\eta = \min \menge{\delta, \delta^\prime}\), and let \(y \in [0,1]\) such that \(\abs*{y-x} < \eta\). Either \(\abs*{f(x)-f(y)} < \varepsilon\), or \(f(x) \neq f(y)\). In the latter case we can decide whether \(x \leqslant y\) or \(x \geqslant y\), so we also have \(\abs*{f(x)-f(y)} < \varepsilon\). 
\end{proof}

We say that a function \(f:[0, 1]\to\RR\) has \define{left limits} if for every \(x\in (0,1]\) there exists \(L \in \RR\) such that for all \(\varepsilon > 0\) there exists \(\delta >0\) such that \(\abs*{f(z) - L} < \varepsilon\) whenever \( z  \in [0,1]\) with \(x - \delta \leqslant z \leqslant x\). We define having \define{right limits} similarly. Functions that have both left and right limits are called \define{regulated} in \cite{dB01}. In the same paper it is shown that an increasing, strongly extensional function \(f:[0, 1]\to\RR\) is regulated. We will improve upon this result by removing the assumption of strong extensionality.

\begin{Proposition} \label{Pro:limits}
An increasing function \(f:[0, 1]\to\RR\) has left and right limits. 
\end{Proposition}

\begin{proof}
We consider only the case of right limits for an increasing function \(f\). For convenience we extend \(f\) to \([0, 2]\) by considering 
\[ f\left(\min \{1,x \}\right) \ + \ \max \{0,x-1\} \ . \]

 Further, for any \(x\in[0,1)\) the right limit of the original function at \(x\) is the same as the right limit of the extended function at \(x\). Fix \(x \in [0, 1)\). We construct a Cauchy sequence \((y_n)_{n\geqslant 1}\) of real numbers converging to the right limit of \(f\) at \(x\), together with a binary sequence \((\lambda_n)_{n\geqslant1}\)  that helps guide the construction. Let \(\lambda_1 = 0\) and \(y_1 = f(x)\). 

Suppose that we have constructed our two sequences up to  but excluding \(n\) and let \(\varepsilon_n = 1 / 2^{-n}\). If \(\lambda_{n - 1} = 1\), then we set \(\lambda_n = 1\) and \(y_n = y_{n - 1}\). If \(\lambda_{n - 1} = 0\) we have more work to do. Applying Proposition \ref{Pro:IVTast} to \(t \mapsto f(t) - f(x) + \nicefrac{\varepsilon_n}{ 2}\) we can either construct \(x_n \in [0, 2]\) such that \(\abs*{f(x_n) - \left(f(x) + \nicefrac{\varepsilon_n}{2}\right)} < \nicefrac{\varepsilon_n}{2}\) or \WLPO holds. In the first case we set \(\lambda_n = 0\) and \(y_n = f(x_n)\) and in the second case we set \(\lambda_n = 1\) and we let \(y_n\) be the Dedekind real given by 
 \[
  \left(\set{q}{\fa{n}{ f\left(x+\frac{1}{n}\right) \geqslant q}} \ , \  \set{q}{\neg\fa{n}{ f\left(x+\frac{1}{n}\right) \geqslant q}}\right);
 \]
\WLPO suffices to define \(y_n\). Note that \(y_n\) is the right limit of \(f\) at \(x\). It is easy to see that \((y_n)_{n\geqslant 1}\) is Cauchy and hence converges to some \(y\in \RR\). If \(y\) is not the right limit of \(f\) at \(x\), then we must have that \(\lambda_n = 0\) for all \(n\). But then \(y = f(x)\) and \((x_n)_{n\geqslant1}\) is a sequence of points greater than \(x\) such that \(f(x_n) \to f(x)\) as \(n\to\infty\), which demonstrates that \(y\) is indeed the right limit of \(f\) at \(x\). Thus  \(y\) is the right limit of \(x\).
\end{proof}

\begin{Corollary} \label{Cor:epssteps}
	Assume \(f:[0,1] \to \RR\) is increasing. For every \(\varepsilon > 0\) there exists \(0 =x_1 < x_2 < \dots < x_n = 1\) such that for all \( 1 \leqslant i< n\) and for all \(x,y \in (x_i, x_{i+1})\)
	\[ \abs*{f(x)-f(y)} \leqslant \varepsilon \ . \] That is only at \(x_i\) can \(f\) make ``\(\varepsilon\)-steps''. Moreover either \( f \) has no \(\varepsilon\)-steps or \( f \) has a discontinuity, in which case \WLPO holds.
\end{Corollary}

\begin{proof}
Let \(y_1, \ldots, y_n\) be real numbers such that

\begin{itemize}
\item \( f(0) < y_1 < y_2 < \cdots < y_n < f(1) \) and
\item \(y_1 - f(0) < \nicefrac{\varepsilon}{2}\), \(f(1) - y_n < \nicefrac{\varepsilon}{2}\), and \( y_{i + 1} - y_i < \nicefrac{\varepsilon}{2} \) for all \( 1 \leqslant i < n \).
\end{itemize}
Applying Proposition \ref{Pro:IVTast} \(n\) times, construct \( x_1, \ldots, x_n \) such that 
 \[
  \abs*{f(x_i) - y_i} > \nicefrac{\varepsilon}{5}  \ \ \implies \ \ f \text{ is discontinuous at }x_i 
 \] 
for each \( 1 \leqslant i \leqslant n \). Since \( f \) is increasing and by the construction in Proposition \ref{Pro:IVTast}, if \( \abs*{f(x_i) - y_i} > \nicefrac{\varepsilon}{5} \), then \( f \) is less than \( y_i + \nicefrac{\varepsilon}{5} \) on \( [0, x_i) \) and greater than \( y_i - \nicefrac{\varepsilon}{5} \) on \( [x_i, 1] \).

Consider \( x, y \in (x_i, x_{i + 1}) \) with \( x \leqslant y \). We have four cases.
 \begin{enumerate}
  \item \( \abs*{f(x_j) - y_j} < \nicefrac{\varepsilon}{4} \) for \( j = i, i + 1 \):
    \begin{align*}
      f(y) - f(x) & \leqslant f(x_{i + 1} - x_i) \\ & < \abs*{f(x_{i + 1}) - y_{i + 1}} + y_{i + 1} - y_i + \abs*{y_i - f(x_i)} \\ &  < \nicefrac{\varepsilon}{4} + \nicefrac{\varepsilon}{2} + \nicefrac{\varepsilon}{4} = \varepsilon.
    \end{align*}
  \item \(\abs*{f(x_i) - y_i} > \nicefrac{\varepsilon}{5} \) and \( \abs*{f(x_{i + 1}) -y_{i + 1}} < \nicefrac{\varepsilon}{4} \):
    \begin{align*}
      f(y) - f(x) & \leqslant f(x_{i + 1}) - y_i  + \nicefrac{\varepsilon}{5} \\ &  = \abs*{f(x_{i + 1}) - y_{i + 1}} +  y_{i + 1} - y_i + \nicefrac{\varepsilon}{5} \\ &  < \nicefrac{\varepsilon}{4} + \nicefrac{\varepsilon}{2} + \nicefrac{\varepsilon}{5} < \varepsilon.
    \end{align*}
  \item \( \abs*{f(x_{i}) -y_{i}} < \nicefrac{\varepsilon}{4} \) and \( \abs*{f(x_{i + 1}) - y_{i + 1}} > \nicefrac{\varepsilon}{5} \):
    \begin{align*}
      f(y) - f(x) & \leqslant y_{i + 1} + \nicefrac{\varepsilon}{5} - f(x_i)  \\ &  = y_{i + 1} - y_i + \abs*{y_i - f(x_i)} + \nicefrac{\varepsilon}{5} \\ &  < \nicefrac{\varepsilon}{2} + \nicefrac{\varepsilon}{4}  + \nicefrac{\varepsilon}{5} < \varepsilon.
    \end{align*}
 \item \( \abs*{f(x_j) - y_j} > \nicefrac{\varepsilon}{5} \) for \( j = i, i + 1 \):
   \begin{align*}
     f(y) - f(x) & \leqslant y_{i + 1} + \nicefrac{\varepsilon}{5} - y_i + \nicefrac{\varepsilon}{5}  \\ &  < \nicefrac{\varepsilon}{2} + \nicefrac{\varepsilon}{5}  + \nicefrac{\varepsilon}{5} \\ &   < \varepsilon.
   \end{align*}
 \end{enumerate}
So in all cases \( f(y) - f(x) < \varepsilon \). Now for arbitrary \( x, y \in (x_i, x_{i + 1}) \) we have that \( \abs*{f(y) - f(x)} \leqslant \varepsilon \) for all  \( x, y \in (x_i, x_{i + 1}) \), since \( \neg\neg(x \leqslant y \lor y \leqslant y) \).

Either \( \abs*{f(x_i) - y_i} < \nicefrac{\varepsilon}{2} \) for all \( 1 \leqslant i \leqslant n\), in which case \( f \) has no \(\varepsilon\)-steps, or there exists \( i \) such that \( \abs*{f(x_i) - y_i} > \nicefrac{\varepsilon}{5} \) and \( f \) is discontinuous at \( x_i \). In the latter case \WLPO holds by Lemma \ref{Lem:DiscImplOmniscience}.
\end{proof}

\begin{Theorem}\label{Thm:discont}
	Assume \(f:[0,1] \to \RR\) is increasing. Then there exists a sequence \( (\xi_n)_{n \geqslant 1} \) in \( [0, 1] \cup \{*\} \) such that \( \set{\xi_n}{\xi_n \in [0, 1]} \) is precisely the set of discontinuities of \( f \) in that 
	\begin{enumerate*}
  \item \label{discont_i} every \( \xi_n \in [0, 1] \) is a discontinuity of \( f \) and
  \item \label{discont_ii} if \( x \) is a discontinuity of \( f \), then there exists \( n \) such that \( x = \xi_n \).
\end{enumerate*}
\end{Theorem}

\begin{proof} 
For each \( \varepsilon_k = 1 / 2^k \  (k \geqslant 1) \) construct, using the previous corollary, \( x^k_1, \ldots, x^k_{m_k} \) such that if \( x \in [0, 1] \) is different form each \( x_i \), then \( x \) has a neighbourhood on which \( f \) varies by less than \( \varepsilon_k \). Let \( (k_n, i_n)_{n \geqslant 1} \) be an enumeration of \( \bigcup \{k\} \times \{1, \ldots, m_k\} \). We construct \( (\xi_n)_{n \geqslant 1} \) as follows. Either \( y^+ - y^- < \varepsilon_{k_n} \) or \( y^+ - y^- > 0 \) where \( y^-, y^+ \) are the left and right limits, respectively, of \( f \) at \( x^{k_n}_{i_n} \). In the first case we set \( \xi_n = * \) and in the second we set \( \xi_n = x^{k_n}_{i_n} \).

By construction, if \( \xi_n \in [0, 1] \) for some \( n \), then \( \xi_n \) is a point of discontinuity of \( f \). Now suppose that \( x \in [0, 1] \) is a point of discontinuity for \( f \). Then there exists \( k \) such that \( f \) has an \( \varepsilon_k \) discontinuity at \( x \) and thus \( \neg\lnot x\in S\) where \( S \) is the finite set \( \{\xi_n: k_n = k\} \). Since each point of \( S \) is an \( \varepsilon_k \) discontinuity of \( f \) and \( f \) is increasing, equality on \( S \) is decidable and we can find \( n \) such that \( x = \xi_n \).
%\( x = x^k_i \) for some \( i \), and so \( \xi_n = x \) where \( n \) is such that \( (k_n, i_n) = (k, i) \).
\end{proof}

The above results easily generalise to a larger class of functions. Let \(\mathbb{F}\) be the smallest class of functions \([0, 1]\to\RR\) that contains all increasing functions and is closed under addition and scalar-multiplication; in other words \(\mathbb{F}\) is the span of increasing functions in the vector space of all functions \([0, 1]\to\RR\). 

 If \( P: a = x_{0} \leqslant x_{1} \leqslant \cdots \leqslant x_{n} = b \) is a partition of \(\left[a, b\right]\), then the corresponding polygonal approximation to \( f \) has variation
 \[
  v_{f, P} \equiv \sum_{i=1}^{n-1} \abs{f(x_{i + 1} - f(x_i)}.
 \]
We say that \( f \) has \define{finite variation} if its variation,
 \[
  \sup\left\{  v_{f,P} : P\text{ is a partition of }[0, 1]  \right\},
 \]
exists.\footnote{In \cite{dB00c} what we call having finite variation is simply called having a variation. Having finite variation should not be confused with having \define{bounded variation}, which is the requirement that the set mentioned is bounded. The two notions differ constructively, as shown in \cite{fR02}.} Every function \( f:[0, 1] \to\RR \) with finite variation can be decomposed into the sum of an increasing function and a decreasing function \cite[Corollary 8]{fR02}. That means that \(\mathbb{F}\) contains all functions of finite variation. However, in \cite{dB16} it is shown that if for every \(  g, h:[0, 1]\to\RR \) such that  \( g, h \) are increasing the function \( F = g - h \) has finite variation, then \LPO holds. Thus, at least constructively, the class of functions of finite variation is not closed under addition which means that the class \(\mathbb{F}\) is a true generalisation of it.

It is trivial that 
\begin{Corollary}
Any function in \(\mathbb{F}\) has left and right limits. 
\end{Corollary}
More interestingly, we can also generalise Theorem \ref{Thm:discont} to \(\mathbb{F}\).
\begin{Corollary}
If \( f \in \mathbb{F} \), then there exists a sequence \( (\xi_n)_{n \geqslant 1} \) in \( [0, 1]\cup\{*\} \) such that 
\begin{enumerate*}
  \item \label{discontcor_i} every \( \xi_n \in [0, 1] \) is a discontinuity of \( f \) and
  \item \label{discontcor_ii} if \( x \) is a discontinuity of \( f \), then there exists \( n \) such that \( x = \xi_n \). 
\end{enumerate*}
\end{Corollary}

\begin{proof}
The proof proceeds by induction on the construction of \( f \in \mathbb{F} \), the base case of which is given by Theorem \ref{Thm:discont}. 

Let \( f = kg \) for some \( k \in \RR \) and \( g \in \mathbb{F} \). If \( k \neq 0 \), then the sequence \( (\xi_n)_{n \geqslant 1} \) enumerating only the discontinuities of \( g \) and \( \{*\} \), that we can construct using our induction hypothesis, also enumerates only the discontinuities of \( f \) and \( \{*\} \). For an arbitrary \( k \in \RR \) we proceed as follows. Fix, using countable choice, an increasing binary sequence \( (\lambda_n)_{n \geqslant 1} \) such that \( k \neq 0 \) if and only if there exists \( n \) such that \( \lambda_n = 1 \); we extend \( \lambda \) to \( \NN \) by setting \( \lambda_0 = 0 \). Then the sequence \( (\zeta_n)_{n \geqslant 1} \) given by
 \[
  \zeta_n = \begin{cases}
                 *                  & \text{if } \lambda_n = 0  \\
                 \xi_{(n - m + 1)}     & \text{if } \lambda_m = 1 - \lambda_{m - 1} \text{ and } n \geqslant m 
               \end{cases} \ 
 \]
enumerates only the discontinuities of \( f \) and \( \{*\} \).

Let \( f = g + h \) where \( g, h \in \mathbb{F} \). Applying our induction hypothesis to both of \( g, h \) we construct sequences \( (\xi^0_n)_{n \geqslant 1}, (\xi^1_n)_{n \geqslant 1} \) that precisely enumerate the discontinuities of \(g, h \), respectively, together with \( \{*\} \). Let 
 \[
  \zeta_n = \begin{cases}
                 \xi^0_{n / 2}       & \text{if } n \text{ is even} \\
                 \xi^1_{(n + 1) / 2} & \text{if } n \text{ is odd} \ .
               \end{cases} 
 \] 
If \( x \) is a discontinuity of \( f \), then \( x \) is a discontinuity of one of \( g, h \). Thus there exist \( i, m \) such that \( x = \zeta_m^i \), so \( x = \zeta_n \) for \( n = 2m + i \). It remains to refine \( (\zeta_n)_{n\geqslant 1}\) to satisfy (\ref{discontcor_i}), as well as (\ref{discontcor_ii}). Let \( (k_n, \ell_n)_{n \geqslant 1} \) be an enumeration of \( \NN^+ \times \NN^+ \) and, using countable choice, construct a binary sequence \( (\lambda_n)_{n \geqslant 1}\) such that
\begin{align*}
	\lambda_n = 0 & \implies \abs*{f(\zeta_{k_n}^-) - f(\zeta_{k_n}^+)} < \frac{1}{\ell_n}\ \ , \\
	\lambda_n = 1 & \implies \abs*{f(\zeta_{k_n}^-) - f(\zeta_{k_n}^+)} > \frac{1}{2\ell_n} \ ,
\end{align*}
where \( f(x^-), f(x^+) \) are the left and right limits, respectively, of \( f \) at \( x \). Define \( (\xi_n)_{n\geqslant 1} \) by 
 \[
  \xi_n = \begin{cases}
                 *          & \text{if } \lambda_n = 0\\
                 \zeta_{k_n}  & \text{if } \lambda_n = 1\ .
               \end{cases} 
 \] 
Then (\ref{discontcor_i}) and (\ref{discontcor_ii}) hold for \((\xi_n)_{n\geqslant 1} \).
\end{proof}

\section{Decomposing Sets.}
It is well known that, constructively, we cannot prove
\[ \fa{a \in \RR}{a \leqslant 0 \ \lor \ 0 < a }  \ , \]
which means for no \(a \in \RR\) can we decompose \(\RR\) into \((-\infty,a] \cup (a, \infty)\)  (also see \cite{iL08}). Of course, it is just as well known that if \(a<b\), then we do have \[ \fa{x \in \RR}{x \leqslant b \ \lor \ 0 < a }  \ , \] which means \(\RR = (-\infty,b] \cup (a, \infty)\). We can turn this situation around to prove the following impossible theorem.
\begin{Proposition} \label{Pro:DecSets}
If \[ \RR = (-\infty,b ] \cup (a, \infty) \]	
then either   \(a < b \) or \( a = b\). \\
Furthermore in case the second alternative holds, \LPO holds.
\end{Proposition}
\begin{proof}
	Consider \(z = b + (b-a)\). If \(z \in (a, \infty)\) then \(2b-a > a\), which means that \(2b >2a\), that is \(b > a\). If \(z \in (-\infty,b]\) then \(2b-a \leqslant b\), which means that \(b \leqslant a\). It is also easy to see that \(a \leqslant b\) (since the assumption that \(a > b\) leads to a contradiction), so, if \(z \in (-\infty,b]\), then \(a = b\). It is well known that if \(a=b\), then \LPO holds.
\end{proof}

The previous result is actually a special case of a more general theorem, whose proof, however, is no longer algebraic and requires countable choice.

\begin{Proposition}
If \(Q\) be a located\footnote{A subset \(A\) of a metric space \((X,\rho)\) is \define{located} if the distance
\[ \rho(A,x) = \inf \set{\rho(a,x)}{a \in A} \]
exists for all \(x\in X\). In CLASS all non-empty sets are located, but constructively that is not always guaranteed.} subset of a complete enough metric space \((X, \rho)\) and \(x \in X\) is such that \[ \fa{z \in X}{z \neq x \ \lor z \notin Q} \ , \]
then either \(\rho(Q,x) > 0\) or \(\rho(Q,x) = 0\). \\ Furthermore, in case the second alternative holds, \LPO holds.
\end{Proposition}

\begin{proof}
Using countable choice and the locatedness of \(Q\), fix a binary sequence \( (\lambda_n)_{n \geqslant 1} \) and a sequence \((q_n)_{n \geqslant 1} \) in \(Q\), such that 
\begin{align*}
	\lambda_n = 0 & \implies \rho(q_n,x) < \frac{1}{2^n} \ , \\
	\lambda_n = 1 & \implies \rho(Q,x) > \frac{1}{2^{n+1}} \ .
\end{align*}
Then \( (\lambda_n)_{n \geqslant 1} \) is increasing and if \(\lambda_n = 0\), then \(\rho(Q,x) < \nicefrac{1}{2^n}\).

Now define a sequence \( (x_n)_{n \geqslant 1} \) by
\begin{align*}
	\lambda_n = 0 & \implies x_n = x \ , \\
	\lambda_n = 1 & \implies x_n = q_m, \text{ where } \lambda_{m}=1- \lambda_{m+1}  \ .	
\end{align*}

The sequence \((x_n)_{n \geqslant 1}\) is easily seen to be a Cauchy sequence. It therefore converges to a limit \(y \in X\).

Notice that if \( \lambda_N = 0 \), then for all \(i \geqslant N\) we have \(\rho(x_i,x) < \nicefrac{1}{2^N}\): either \(\lambda_i = 0\) and so \(\rho(x_i,x) = \rho(x,x) = 0 \), or \(\lambda_i = 1\), which means that \(x_i = q_m\) for some \(i > m \geqslant N \), and therefore \[ \rho(x_i,x) = \rho(q_m,x) < \frac{1}{2^m} \leqslant \frac{1}{2^N} \ .\]
That means that 
\begin{equation} \label{Eqn:decomplimit}
	\lambda_N = 0 \implies \rho(y,x) \leqslant \frac{1}{2^N} \ .
\end{equation}

Now either \(y \notin Q \) or \(y \neq x\). In the first case we must have \( \lambda_n = 0 \) for all \(n \in \NN\), and therefore \(\rho(Q,x) = 0\). In the second case there exists \( N\) such that \( \rho(x,y) > \nicefrac{1}{2^N} \).  For that \(N\) we must have \( \lambda_N = 1 \), since, as shown in \eqref{Eqn:decomplimit}, \( \lambda_N = 0 \) implies that  \( \rho(x, y) \leqslant  \nicefrac{1}{2^N} \).

For the rest of the proof assume that \(\fa{n \in \NN}{\lambda_n=0}\). To see that \LPO holds let \( (a_n)_{n \geqslant 1}\) be an increasing binary  sequence. Define \((z_n)_{n \geqslant 1}\) by
\begin{align*}
	a_n = 0 & \implies z_n = q_n \ , \\
	a_n = 1 & \implies z_n = q_m, \text{ where } a_{m}=1- a_{m+1}  \ .	
\end{align*}
Again, this is easily seen to be a Cauchy sequence; let \(z\) be its limit. By our assumption either \(z \neq x\) or \(z \notin Q\). In the second case there cannot be \(n \) such that \(a_n=1\), since then \( z = q_m \) for some \(m \leqslant n\) and therefore \(z \in Q\); thus in the second case \(\fa{n \in \NN}{a_n=0}\). In the first case, choose \(n\) such that \(\rho(x,z) > \nicefrac{1}{2^n}\). We must have \(a_n = 1\), since if \(a_n=0\) we have \(\rho(z_n, x) =  \rho(q_n, x) < \nicefrac{1}{2^n}\) and therefore, after taking the limit of the left hand side, \(\rho(z,x) \leqslant \nicefrac{1}{2^n}\).
Thus \LPO holds. 
\end{proof}

\section{Neatly Located}

In \cite{hD08a} and \cite[Chapter 3]{hD08b} the notion of a neatly located set in an apartness space is introduced, which underlies another impossible theorem. Formulated for metric spaces the definition is as follows.

\begin{Definition}
We say that two subsets \(S,T\) of a metric space \((X, \rho)\) are, or form, a \define{neat covering} if there exist \( \varepsilon>0\) and \(S^{\prime},T^{\prime}\) such that 
\[ \fa{s \in S^{\prime}, t \in T^{\prime} }{\rho(s,t)} > \varepsilon \ , X=T\cup T^{\prime}, \text{ and }X=S\cup S^{\prime} \ . \]
The following diagram might be more illustrative than the definition.
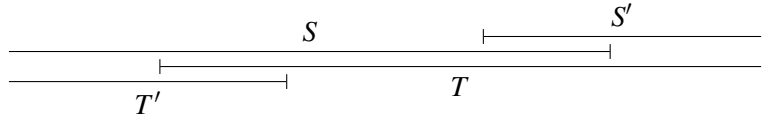
\begin{center}
\begin{tikzpicture}[y = 0.4cm]
\draw[|-] (6.3,0) -- node[above] {\(S^\prime\)} (10,0);
\draw[-|] (0,-0.5) -- node[above] {\(S\)} (8,-0.5);
\draw[|-] (2,-1) -- node[below] {\(T\)} (10,-1);
\draw[-|] (0,-1.5) -- node[below] {\(T^\prime\)} (3.7,-1.5);
\end{tikzpicture}
\end{center}

We call an inhabited subset \(A\) of  \(X\) \define{neatly located} if for all neat coverings \(S,T\) of \(X\), either \(A\subset S\) or \(A\cap T\neq\emptyset\).	
\end{Definition}

Classically every set is neatly located, which is a trivial consequence of \LEM. However, constructively, \(\RR\) is neatly located if and only if \LPO holds \cite[Proposition 2.8]{hD08a}; on the other hand one can show that all totally bounded sets are neatly located \cite[Proposition 2.5]{hD08a}. Again, we can turn this situation on its head and  prove an impossible theorem \cite[Proposition 2.9]{hD08a}.
\begin{Proposition}
If \(A \) is a neatly located subset of a separable metric space \((X,\rho) \) and 	\(\varepsilon >0 \), then either there is a finite \(\varepsilon\)-approximation to \(A\), or there exists a sequence \(x_n\) in \(A\) such that \(\fa{i,j}{\rho(x_i,x_j) > \nicefrac{\varepsilon}{2}}\). \\
Furthermore in case the second alternative holds, \LPO holds.
\end{Proposition}

Very loosely interpreted, this says that in practice, a neatly located set is either totally bounded or it is not and \LPO holds.

It is also worth pointing out that neat locatedness behaves very canonically. For example it is easily seen to be preserved under strongly continuous mappings, which is a form of continuity lying strictly between point-wise and uniform continuity \cite{dB06}.

\section{Positive Infima of Convex Functions}

In \cite[Theorem 1]{jB16} it is shown that a uniformly continuous, quasi-convex function \(f : C \to R^+\) defined on a convex and compact subset \(C\) of \( \RR^n \) has a positive infimum. 
Here a function is called \define{quasi-convex} if for all \(x, y\in X\) and every \(\lambda\in[0,1]\)
  \[
  f(\lambda x+(1-\lambda)y) \leqslant \max\{f(x), f(y)\} \ .
  \] 
This is an interesting result in constructive mathematics, since without the assumption of quasi-convexity the statement lies on the demarcation line between RUSS on the one side and INT on the other side. In the former there  actually exists a uniformly continuous function \(f : [0,1] \to R^+\) with \(\inf f = 0\) (see for example \cite{Beeson1985}). In INT the statement is true as a consequence of the fan theorem for decidable bars \cite{dB87}.

At least for the unit interval, we can actually improve \cite[Theorem 1]{jB16} and replace uniform continuity by the assumption that the infimum of each subinterval exists. What we get is, yet another, impossible theorem.

\begin{Proposition} \label{Pro:impconvex}
Let \(f:[0, 1]\to\RR^+\) be a quasi-convex function such that \(\inf\set{f(x)}{x\in I}\) exists for each sub-interval \(I\) of \([0, 1]\). Then either the infimum of \(f\) on \([0, 1]\) is positive or \(\inf f = 0 \). \\
Furthermore in case the second alternative holds \(f\) has a discontinuity and \WLPO holds.
\end{Proposition}

\begin{proof}
We inductively define a sequence \((J_n)_{n \geqslant 1}\) of closed subintervals of \([0,1]\) and an increasing binary sequence \((\lambda_n)_{n \geqslant 1}\) such that for all \(n \in \NN\)
\begin{enumerate}
  \item \(J_{n} \subset J_{n-1}\), for $n \neq 0$,
  \item \(\abs*{J_{n}} \leqslant \nicefrac{1}{2^n}\),
  \item \(\inf f([0,1]) = \inf f(J_n) \), and
  \item \(\lambda_n = 1 \implies J_n = [m,m] \land \inf f([0,1]) \geqslant \frac{f(m)}{2}\).
\end{enumerate}
We begin the construction by setting \(J_0 = [0, 1]\) and \(\lambda_0 = 0\). Now suppose we have constructed both sequences up to \(n\). If \(\lambda_n = 1\) we simply set \(\lambda_{n+1} = \lambda_n\) and \(J_{n+1} = J_n\). If \(\lambda_n = 0\) we consider the midpoint \(m_n\) of the interval \(J_n\) and the left and right half subintervals \(I_\ell\) and \(I_r\). Since \(f(m_n) > 0\) we can continue in two cases which are that either \(\inf f(I_\ell) > f(m_n)/2\) or  \(\inf f(I_\ell) < f(m_n)\). 
\begin{itemize}
  \item \(\inf f(I_\ell) > f(m_n)/2\): In that case we also check whether \(\inf f(I_r) > f(m_n)/2\) or \( \inf f(I_r) < \inf f(I_\ell)\). In the second case we have that \(\inf f(J_n) = \inf f(I_r)\), and we can set \(\lambda_{n+1}=0\), and \(J_{n+1} = I_r\). In the first case we have that \[ \inf f(J_n) = \min \menge {\inf f(I_\ell), \inf f(I_r)} > \frac{f(m_n)}{2} \ . \] That means we can set \(\lambda_{n+1} = 1\) and \(J_{n+1} = [m_n,m_n]\).
  \item \(\inf f(I_\ell) < f(m_n)\):  In that case we also check whether \(\inf f(I_r) > \inf f(I_\ell)\) or \( \inf f(I_r) < f(m_n)\). But the second case can be ruled out, since it implies that there is \(a \in I_\ell\) and \(b \in I_r\) such that \(f(a) < f(m_n)\) and \(f(b) < f(m_n)\), which is a contradiction to the quasi-convexity of \(f\). Thus \(\inf f(I_r) > \inf f(I_\ell)\), so \(\inf f(J_n) = \inf f(I_\ell)\) and we can set \(\lambda_{n+1}=0\), and \(J_{n+1} = I_\ell\).
\end{itemize}
This concludes the construction of the sequences \((J_n)\) and \((\lambda_n)\). 

By construction the sequence of the left-endpoints (or the right, or any) is a Cauchy sequence, which therefore converges to a limit \(z\). Again, \(f(z) > 0\), so we can check whether \(\inf f([0,1]) > 0\), or \(\inf f([0,1]) < \nicefrac{f(z)}{2}\).  In the first case we are done, so for the rest of the proof we focus on the second case. In that case there cannot be \(n\) such that \(\lambda_n = 1\), since that would imply that there is an \(m\) such that  \(J_n = [m,m]\), and \(\inf f([0,1]) \geqslant f(\nicefrac{m}{2})\). But since \(z \in J_n\) this means that \(\inf f([0,1]) \geqslant \nicefrac{z}{2}\); a contradiction. Thus \(\fa{n \in \NN}{\lambda_n = 0}\).

We construct, using countable choice, an increasing binary sequence \((\alpha_n)_{n \geqslant 1}\) and a sequence \((y_n)_{n \geqslant 1}\) in \( [0, 1]\) such that
\begin{align*}
	\alpha_n = 0 & \implies  \inf f([0,1]) < \frac{1}{2^n} \ , \\ 
	\alpha_n = 1 & \implies  \inf f([0,1]) > \frac{1}{2^{n+1}},  \ 
\end{align*}
\(y_n \in J_n\), and \(\alpha_n = 0 \implies y < \nicefrac{1}{2^n}\) for every \(n \in \NN\). It is obvious that \(y_n \to z\). However, we now consider the limit \(y\) of the sequence \((a_n) \ast (y_n)\). Very similar to what we did above, since \(f(y) > 0\), we have  \(0 < \inf f([0,1])\) and we are done, or \( \inf f([0,1]) < \nicefrac{f(y)}{2} \). In the second case we must have \(\fa{n \in \NN}{\alpha_n = 0}\). For suppose there is \(n\) such that \(\alpha_n = 1\). Then we can find \(m < n\) such that \(\alpha_m = 0\) and \(\alpha_{m+1} = 1\). So \((a_n) \ast (y_n)\) will from \(m\) on be constant \(y_m\), and so \(y = y_m\). But this yields the contradiction 
\[ \frac{1}{2^{m+1}} < \inf f([0,1]) < \frac{f(y)}{2} = \frac{f(y_m)}{2} < \frac{1}{2^{m+1}}  \ . \] 
The fact that \(\fa{n \in \NN}{\alpha_n = 0}\) implies  \(\inf f([0,1]) = 0\), but also that \(y_n \to z\) and \(f(y_n) \to 0\). Since \(f(z) > 0\) this means that \(f\) is discontinuous, and therefore (see Lemma \ref{Lem:DiscImplOmniscience}) \WLPO holds. 
\end{proof}

As already mentioned above we cannot hope to remove the quasi-convexity, since in RUSS there is a positive, uniformly continuous function $f$ on \([0, 1]\) with  \(\inf f = 0\). We also cannot drop the other assumption.  
\begin{Remark} One cannot remove the assumption of inf-locatedness from Proposition \ref{Pro:impconvex}, since in RUSS the statement that for every  quasi-convex  \(f:[0, 1]\to\RR^+\) there exist \(\varepsilon >0\) such that \(\fa{x\in[0,1]}{f(x) > \varepsilon}\) implies \LPO.
\end{Remark}
\begin{proof}
Let \((r_n)_{n \geqslant 1}\) be a Specker sequence\footnote{A Specker sequence is an (increasing) sequence of rationals in \([0,1]\) that is eventually bounded away from every point in \([0,1]\)---that is for every \(x \in [0,1]\) there exists \(n\) such that \(\abs*{r_i - x} > \frac{1}{2^n}\) for all \(i \geqslant n\). A Specker sequence is a strong counterexample to the Bolzano-Weierstra\ss\xspace Theorem. See \cite[Chapter 3, Theorem 3.1]{dB87} or the original paper \cite{Specker1949} for details.} and \((a_n)_{n \geqslant 1}\) a binary sequence with at most one \(1\). Now let \[ t_{z,\varepsilon}(x) = \max \menge{ 0, 1- \frac{1}{\varepsilon} \abs*{x-z} } \] be the spike function around \(z\) with width \(2 \varepsilon \), and define 
\[f = 1 - \sum a_n \left(1- \frac{1}{n}\right) t_{r_n,\frac{1}{2^n}} \ .  \]
The function \(f\) is either constant \(1\) if \(a_n =0\) for all \(n \in \NN\) or looks like below if \(m\) is such that \(a_m =1 \). Of course, constructively we do not know a priori which of these cases we are in.

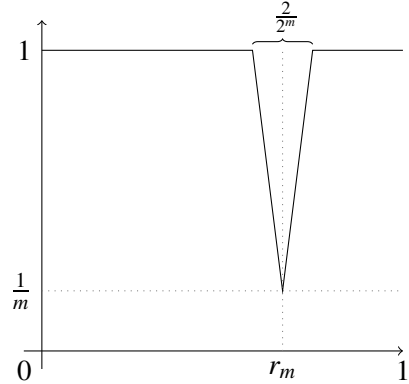
\begin{center}
\begin{tikzpicture}[scale=0.8]
\draw[->] (-0.3,0) -- (6,0);
\draw[->] (0,-0.3) -- (0,5.5);
\path (6,0) node[anchor=north] {\(1\)};
\path (4,0) node[anchor=north] {\(r_m\)};
\path (0,0) node[anchor=north east] {\(0\)};
\path (0,5) node[anchor=east] {\(1\)};
\path (0,1) node[anchor=east] {\(\frac{1}{m}\)};
\draw (0,5) -- (3.5,5) -- (4,1) -- (4.5,5) -- (6,5);
\draw[dotted, gray] (0,1) -- (6,1);
\draw[dotted, gray] (4,0) -- (4,5);
\draw [decorate, decoration={brace,amplitude=2pt}] (3.5,5.1) -- (4.5,5.1) node [black,anchor=south, xshift=-10pt, yshift=1pt] {\footnotesize \(\frac{2}{2^m}\)};
\end{tikzpicture}
\end{center}

Note that if we have found \(n\) such that \(a_n =1\), then \(f\) is clearly quasi-convex. More generally let \(a \leqslant b \leqslant c \). Since \((r_n)_{n \geqslant 1}\) is a Specker sequence there exists \(n\) such that \( \abs*{x - r_i} > \frac{1}{2^n} \) for all \(i \geqslant n\) and \(x \in  \menge{a,b,c}\). Now either there is \(j \leqslant n\) such that \(a_j=1\) or not. In the first case we are done by what was said above. In the second case \(f(a) = f(b) = f(c) =1 \), since \( t_{r_i,\frac{1}{2^i}} (x) = 0\) for \(i \geqslant n\) and \(x \in  \menge{a,b,c}\). 

So \(f\) is quasi-convex. Now assume there exist \(\varepsilon >0\) such that \(\fa{x\in[0,1]}{f(x) > \varepsilon}\). Choose \( m \) such that \( \frac{1}{m}<\varepsilon\). There cannot be \(i \geqslant m\) with \(a_i = 1\), since in that case \(f(r_i) = \frac{1}{i} < \frac{1}{m} < \varepsilon \). Hence we can decide 
\[ \fa{n \in \NN}{a_n = 0} \ \lor \ \ex{n \in \NN}{a_n =1}  \ ; \]
thus \LPO holds.
\end{proof}

Following the same lines as \cite{jB16} it seems feasible to extend Proposition \ref{Pro:impconvex} to functions defined on a convex and compact subset \(C\) and thereby obtaining a true generalisation of the main result of that paper.

\section*{Acknowledgements}
This work was supported by a Marie Curie IRSES award from the European Union, with counterpart funding from the Ministry of Research, Science \& Technology of New Zealand, for the project CORCON.

\bibliographystyle{plain}
\bibliography{All}

\end{document}